\documentclass[12pt,twoside]{amsart}
\usepackage{amsmath, amsthm, amscd, amsfonts, amssymb, graphicx}
\usepackage{enumerate}
\usepackage[colorlinks=true,
linkcolor=blue,
urlcolor=cyan,
citecolor=red]{hyperref}
\usepackage{mathrsfs}
\newtheorem{theorem}{Theorem}[section]

\newtheorem{remark}{Remark}[section]
\newtheorem{definition}{Definition}[section]
\newtheorem{corollary}{Corollary}[section]
\newtheorem{example}{Example}[section]
\newtheorem{proposition}{Proposition}[section]
\numberwithin{equation}{section}

\begin{document}
	
\title{Radical convex functions}
\author{Mohammad Sababheh and Hamid Reza Moradi}
\subjclass[2010]{Primary 26A51; Secondary 26D15, 39B62.}
\keywords{Convex function, Jensen’s inequality, Hermite-Hadamard inequality, Hardy inequality.}
\maketitle

\begin{abstract}
In this article, we further explore convex functions by revealing new bounds, resulting from stronger convexity behavior. In particular, we define the so called radical convex functions and study their properties. We will see that such convex functions are bounded above by new curves, rather than straight lines. Applications including discrete and continuous Jensen inequalities, subadditivity behavior, Hermite-Hadamard and integral inequalities will be presented.
\end{abstract}
%------------------------------------------------------------------------------------%
\pagestyle{myheadings}
\markboth{\centerline {}}
{\centerline {}}
\bigskip
\bigskip
%------------------------------------------------------------------------------------%
%------------------------------------------------------------------------------------%

\section{Introduction}
Convex functions and their properties have been in the core of studying Mathematical inequalities. This includes inequalities among real numbers, functional inequalities, probability inequalities and matrix inequalities, to mention a few.

In this article, we will be interested in convex functions $f:[0,\infty)\to[0,\infty)$.
Recall that a convex function is a function that lies under its secants over the interval of convexity. This is equivalent to saying
\begin{equation}\label{eq_conv_def}
f((1-t)a+tb)\leq (1-t)f(a)+tf(b),\;\forall a,b\in [0,\infty), 0\leq t\leq 1.
\end{equation}
This inequality can be extended to $n-$parameters via the so called Jensen's inequality stating that for the positive weights $w_i$ with $\sum_{i=1}^{n}w_i=1,$ one has the inequality
\begin{equation}\label{eq_jensen_itro}
f\left(\sum_{i=1}^{n}w_ix_i\right)\leq \sum_{i=1}^{n}w_if(x_i),\;\forall x_i\in [0,\infty).
\end{equation}
A concave function is a function $f$ such that $-f$ is convex. So, for a concave function \eqref{eq_jensen_itro} is reversed.\\
The inequality \eqref{eq_jensen_itro} has a continuous version stating that
\begin{equation}\label{eq_jensen_cont_intro}
f\left(\frac{1}{b-a}\int_{a}^{b}g(x)dx\right)\leq \frac{1}{b-a}\int_{a}^{b}(f(g(x))dx,
\end{equation}
for the continuous function $g:[a,b]\to [0,\infty)$ and the convex function $f:[0,\infty)\to\mathbb{R}.$\\
An important  inequality is the well known Hermite-Hadamard inequality stating that \cite{mitri,nic}
\begin{equation}\label{eq_HH_intro}
f\left(\frac{a+b}{2}\right)\leq \frac{1}{b-a}\int_{a}^{b}f(x)dx\leq \frac{f(a)+f(b)}{2},
\end{equation}
valid for the convex function $f:[a,b]\to\mathbb{R}$. This inequality refines \eqref{eq_conv_def} when $t=\frac{1}{2}.$

Among the most important properties of concave/convex functions is their sub or super additive behavior. That is, a concave function with $f(0)=0$ satisfies the subadditive inequality \cite[Problem II.5.12]{bhatia}
\begin{equation}\label{eq_conc_sub}
f(a+b)\leq f(a)+f(b),\;a,b\geq 0,
\end{equation}
while a convex function $f$ with $f(0)=0$ satisfies \eqref{eq_conc_sub} with the inequality reversed; as a super additive behavior of convex functions.

In this article, we will treat convex functions looking into their other properties. This approach will allow obtaining new bounds, and nonlinear terms that are related to convex functions. To simplify our statements, we introduce the following simple definition.
\begin{definition}\label{def}
Let $f:[0,\infty)\to\mathbb{R}$ be a continuous  function with $f(0)=0,$ and let $p\geq 1$ be a fixed number. If  the function $g(x)=f\left(x^{\frac{1}{p}}\right)$ is convex on $[0,\infty)$, we say that $f$ is $p-$radical convex.
\end{definition}
\begin{remark}
Although Definition \ref{def} is stated for functions defined on $[0,\infty)$, it can be stated for any interval $[0,\alpha)$ where $\alpha^{\frac{1}{p}}\leq \alpha;$ to guarantee the well definiteness of the quantity $f\left(x^{\frac{1}{p}}\right)$. Since $p\geq 1$, then $\alpha\geq 1$ can be selected arbitrarily. Also, we remark that the assumption $f(0)=0$ is essential, as we will need the super additivity behavior of convex functions, which needs this assumption.
\end{remark}

Before proceeding, we list some basic properties of $p-$radical convex functions.
\begin{proposition}
Let $f$ be $p-$radical convex, for some $p\geq 1.$ 
\begin{enumerate}
\item $f$ is increasing and convex.
\item $f$ is $q-$radical convex, for all $1\leq q\leq p.$
\item If $g$ is $q-$radical convex for some $q\geq 1,$ then $f+g$ is $\min\{p,q\}-$radical convex.
\item If $g$ is increasing convex, then the composite function $g\circ f$ is $p-$radical convex.
\end{enumerate}
\end{proposition}
\begin{proof}
We prove the first and second assertions only. The others are straightforward.
\begin{enumerate}
\item Let $g(x)=f\left(x^{\frac{1}{p}}\right)$. Then $g$ is convex, and $g(0)=0.$ A convex function on $[0,\infty)$ is either decreasing on $[0,\infty)$, increasing on $[0,\infty)$ or decreasing on $[0,\alpha]$ and increasing on $[\alpha,\infty)$ for some $\alpha>0.$ Since $g(0)=0$ and $g\geq 0,$ it follows that $g$ is increasing on $[0,\infty).$ Consequently, $f(x)=g(x^p)$ is the composition of two increasing functions, hence $f$ is increasing. Further, since $g$ is convex increasing and $x\mapsto x^p$ is convex, it follows that $f$ is convex.

\item Since $f$ is $p-$radical convex, the function $g(x)=f\left(x^{\frac{1}{p}}\right)$ is convex. Define $h(x)=f\left(x^{\frac{1}{q}}\right)$, for $1\leq q\leq p.$ It is clear that $h(x)=g\left(x^{\frac{p}{q}}\right)$. By the first statement of the proposition, $f$ is convex increasing, and hence $g$ is increasing. Since the mapping $x\mapsto x^{\frac{p}{q}}$ is convex and $g$ is increasing convex, it follows that $h$ convex.
\end{enumerate}
\end{proof}

We also notice that $p-$radical convex functions can be constructed in different ways. For example, if $f$ is convex with $f(0)=0,$ then the function $g$ defined by $g(x)=f\left(x^p\right)$ is $p-$radical convex. Another observation is that if $f$ has the Maclaurin series
$$f(x)=\sum_{n=0}^{\infty}\alpha_nx^n,\;\alpha_n>0,$$ then the function
$$g(x)=f(x)-\sum_{n=0}^{p^*}\alpha_nx^n,$$ can be easily seen to be $p-$radical convex, where $p^*$ is the greatest integer less than $p$.

\begin{example}
\begin{enumerate}
\item $f(x)=e^x$ has the Maclaurin series $$e^x=\sum_{n=0}^{\infty}\frac{1}{n!}x^n.$$ Since $\frac{1}{n!}>0,$ it follows that
$$e^x-\sum_{n=1}^{p^*}\frac{1}{n!}x^n$$ is $p-$radical convex. For example, $f_1(x)=e^x-1-x$ is $2-$radical convex, $f_2(x)=e^x-1-x-\frac{x^2}{2}$ is $3-$radical convex, and so on.

\item $f(x)=\frac{1}{1-x}$ has the Maclaurin series
$$\frac{1}{1-x}=\sum_{n=0}^{\infty}x^n, 0\leq x<1.$$ We may construct $p-$radical convex functions from $f$ as follows. $f_1(x)=\frac{1}{1-x}-1-x$ is $2-$radical convex, $f_2(x)=\frac{1}{1-x}-1-x-x^2$ is $3-$radical convex, and so on.

\item A similar argument applies to the function $$f(x)=-\ln(1-x)=\sum_{n=0}^{\infty}\frac{1}{n+1}x^{n+1}, 0<x<1.$$
\end{enumerate}
\end{example}

As  easy examples of $p-$radical convex functions, we notice that $f(x)=x^2$ is 2-radical convex but not 3-radical convex, while the function $f(x)=x^4$ is $p-$radical convex for all $1\leq p\leq 4.$ So, although both functions $f_1(x)=x^2$ and $f_2(x)=x^4$ are convex, it seems that the two functions do not behave similarly, in terms of convexity. This understanding will lead to interesting forms of \eqref{eq_conv_def}, \eqref{eq_jensen_itro} and \eqref{eq_conc_sub}. 

For example, we will show that a 2-radical convex function $f$ satisfies the interesting inequality
\begin{equation}\label{eq_2-conv_intro}
f\left( \left( 1-t \right)a+tb \right)+f\left( \sqrt {t(1-t)} \left| a-b \right| \right)\le \left( 1-t \right)f\left( a \right)+tf\left( b \right), a,b\in [0,\infty), 0\leq t\leq 1.
\end{equation}
Since $f\geq 0,$ by definition, the inequality \eqref{eq_2-conv_intro} provides a new refining term for \eqref{eq_conv_def}. Although \eqref{eq_conv_def} has been refined in the literature, the new refinement in \eqref{eq_2-conv_intro} presents a non-linear smooth refining term, namely  $f\left( \sqrt {t(1-t)} \left| a-b \right| \right)$. We refer the reader to \cite{mitroi,sab_mjom,sab_mia} for refinements that include linear or piecewise linear refining terms. Also, we refer the reader to a non-linear refinement of convex functions in \cite{sab_bull}. Then, new forms of the subadditive inequality \eqref{eq_conc_sub} will be presented for such functions. Further applications include new forms of the Hermite-Hadamard inequality and new unexpected bounds for 2-convex functions.

After establishing our results for $2-$radical convex functions, we go over $p-$radical convex functions more generally. We will show multiple terms refining \eqref{eq_conv_def} for $p-$radical convex functions and a new form of the Hermite-Hadamrd inequality.   A nice application of $p-$radical convex functions will be its relation with the celebrated Hardy inequality stating
\begin{equation}\label{eq_hardy_intro}
\int_{0}^{\infty}\left(\frac{1}{x}\int_{0}^{x}f(t)dt\right)^pdx\leq \left(\frac{p}{p-1}\right)^{p}\int_{0}^{\infty}f(x)^pdx,
\end{equation}
valid for the measurable function $f:(0,\infty)\to (0,\infty)$ and $p>1$. When $f$ is $p-$radical convex, it can be easily seen that $\int_{0}^{\infty}f(x)^pdx=\infty.$ However, we will be able to prove a new version of \eqref{eq_hardy_intro}, where the interval $(0,\infty)$ is replaced by any other finite-length interval. This can be seen in Theorem \ref{thm_hardy} below.

\section{2-radical Convex functions}
In this section, we study detailed properties of 2-radical convex functions. First, we present a refinement of Jensen's inequality for $2-$radical convex functions.
\begin{theorem}\label{10}
 Let $f$ be 2-radical convex. If ${{x}_{1}},\ldots ,{{x}_{n}}\ge 0$ and $0\le {{w}_{1}},\ldots ,{{w}_{n}}\le 1$ are such that $\sum\nolimits_{i=1}^{n}{{{w}_{i}}}=1$, then
\[f\left( \sum\limits_{i=1}^{n}{{{w}_{i}}{{x}_{i}}} \right)\le \sum\limits_{i=1}^{n}{{w}_{i}}\left\{f\left( \frac{\sum\nolimits_{j=1}^{n}{{{w}_{j}}{{x}_{j}}}+{{x}_{i}}}{2} \right)+f\left( \frac{\left| \sum\nolimits_{j=1}^{n}{{{w}_{j}}{{x}_{j}}}-{{x}_{i}} \right|}{2} \right)\right\}\le \sum\limits_{i=1}^{n}{{{w}_{i}}f\left( {{x}_{i}} \right)}.\]
\end{theorem}
\begin{proof}
Assume that $0\le t\le 1$. We have
\begin{equation}
\begin{aligned}\label{1}
  & \left( 1-t \right){{a}^{2}}+t{{b}^{2}}-{{\left( \left( 1-t \right)a+tb \right)}^{2}}-{{\left( 1-t \right)}^{2}}{{\left( a-b \right)}^{2}} \\ 
 & =\left( 1-t \right)\left( 2t-1 \right){{a}^{2}}+\left( 1-t \right)\left( 2t-1 \right){{b}^{2}}-2\left( 1-t \right)\left( 2t-1 \right)ab \\ 
 & =\left( 1-t \right)\left( 2t-1 \right)\left( {{a}^{2}}+{{b}^{2}}-2ab \right) \\ 
 & =\left( 1-t \right)\left( 2t-1 \right){{\left( a-b \right)}^{2}}.
\end{aligned}
\end{equation}
That is,
\[\left( 1-t \right){{a}^{2}}+t{{b}^{2}}-{{\left( \left( 1-t \right)a+tb \right)}^{2}}-{{\left( 1-t \right)}^{2}}{{\left( a-b \right)}^{2}}=\left( 1-t \right)\left( 2t-1 \right){{\left( a-b \right)}^{2}}.\]
Thus,
\[\left( 1-t \right){{a}^{2}}+t{{b}^{2}}=t\left( 1-t \right){{\left( a-b \right)}^{2}}+{{\left( \left( 1-t \right)a+tb \right)}^{2}}.\]
Let $g\left( t \right)=f\left( \sqrt{t} \right)$, $t\in \left[ 0,\infty  \right)$. Then $g$ is an increasing convex function on $\left[ 0,\infty  \right)$. This implies,
\[\begin{aligned}
   g\left( {{\left( \left( 1-t \right)a+tb \right)}^{2}} \right)+g\left( {{t(1-t)}}{{\left( a-b \right)}^{2}} \right)&\le g\left( {{\left( \left( 1-t \right)a+tb \right)}^{2}}+{{t(1-t) }}{{\left( a-b \right)}^{2}} \right) \\ 
 & = g\left( \left( 1-t \right){{a}^{2}}+t{{b}^{2}} \right) \\ 
 & \le \left( 1-t \right)g\left( {{a}^{2}} \right)+tg\left( {{b}^{2}} \right).  
\end{aligned}\]
Consequently,
\[g\left( {{\left( \left( 1-t \right)a+tb \right)}^{2}} \right)+g\left( {{t(1-t) }}{{\left( a-b \right)}^{2}} \right)\le \left( 1-t \right)g\left( {{a}^{2}} \right)+tg\left( {{b}^{2}} \right).\]
Thus,
\begin{equation}\label{eq_2-conv_pf}
f\left( \left( 1-t \right)a+tb \right)+f\left( \sqrt {t(1-t)} \left| a-b \right| \right)\le \left( 1-t \right)f\left( a \right)+tf\left( b \right),
\end{equation}
where $0\le t\le 1$.  In particular,
\begin{equation}\label{8}
f\left( \frac{a+b}{2} \right)+f\left( \frac{\left| a-b \right|}{2} \right)\le \frac{f\left( a \right)+f\left( b \right)}{2}.
\end{equation}
Replacing $a$ and $b$ by $\sum\nolimits_{i=1}^{n}{{{w}_{i}}{{x}_{i}}}$ and $x_i$, respectively, in \eqref{8}, we get
\begin{equation}\label{7}
f\left( \frac{\sum\nolimits_{j=1}^{n}{{{w}_{j}}{{x}_{j}}}+{{x}_{i}}}{2} \right)+f\left( \frac{\left| \sum\nolimits_{j=1}^{n}{{{w}_{j}}{{x}_{j}}}-{{x}_{i}} \right|}{2} \right)\le \frac{f\left( \sum\nolimits_{i=1}^{n}{{{w}_{i}}{{x}_{i}}} \right)+f\left( {{x}_{i}} \right)}{2}.
\end{equation}
Multiplying \eqref{7} by ${{w}_{i}}\ge 0$ $\left( i=1,\ldots ,n \right)$ and summing over $i$ from $1$ to $n$ we may deduce
\[\begin{aligned}
  & f\left( \sum\limits_{i=1}^{n}{{{w}_{i}}{{x}_{i}}} \right) \\ 
 & \le \sum\limits_{i=1}^{n}{{{w}_{i}}f\left( \frac{\sum\nolimits_{j=1}^{n}{{{w}_{j}}{{x}_{j}}}+{{x}_{i}}}{2} \right)} \\ 
 & \le \sum\limits_{i=1}^{n}{{{w}_{i}}f\left( \frac{\sum\nolimits_{j=1}^{n}{{{w}_{j}}{{x}_{j}}}+{{x}_{i}}}{2} \right)}+\sum\limits_{i=1}^{n}{{{w}_{i}}f\left( \frac{\left| \sum\nolimits_{j=1}^{n}{{{w}_{j}}{{x}_{j}}}-{{x}_{i}} \right|}{2} \right)} \\ 
 & \le \frac{f\left( \sum\nolimits_{i=1}^{n}{{{w}_{i}}{{x}_{i}}} \right)+\sum\nolimits_{i=1}^{n}{{{w}_{i}}f\left( {{x}_{i}} \right)}}{2} \\ 
 & \le \sum\limits_{i=1}^{n}{{{w}_{i}}f\left( {{x}_{i}} \right)},
\end{aligned}\]
which gives the desired inequality.
\end{proof}
The inequality \eqref{eq_2-conv_pf} is of special interest that it deserves to be mentioned explicitly.
\begin{corollary}\label{cor_2-conv_ref}
Let $f$ be 2-radical convex and let $a,b\geq 0.$ If $0\leq t\leq 1,$ then 
$$f\left( \left( 1-t \right)a+tb \right)+f\left( \sqrt {t(1-t)} \left| a-b \right| \right)\le \left( 1-t \right)f\left( a \right)+tf\left( b \right).$$
\end{corollary}
In particular, the following inequality holds for $2-$radical convex functions.
\begin{corollary}
Let $f$ be $2-$radical convex. Then, for $0\leq t\leq 1,$
$$f(t)\leq f(1) t-f\left(\sqrt{t(1-t)}\right)\leq f(1)t.$$
\end{corollary}
This provides a non linear term bounding the $2-$radical convex function from above, providing a better bound than the linear one.

Further, we have the following application related to the arithmetic-geometric mean inequality.
\begin{corollary}
Let $f(x)=x^p, p\geq 2$ and let $x_i>0$ and $w_i>0$ with $\sum_{i=1}^{n}w_i=1.$  Then 
\[\begin{aligned}
   \prod\limits_{i=1}^{n}{x_{i}^{{{w}_{i}}}}&\le \sum\limits_{i=1}^{n}{{{w}_{i}}\left\{ f\left( \frac{\sum\nolimits_{j=1}^{n}{{{w}_{j}}{{f}^{-1}}\left( {{x}_{j}} \right)}+{{f}^{-1}}\left( {{x}_{i}} \right)}{2} \right)+f\left( \frac{\left| \sum\nolimits_{j=1}^{n}{{{w}_{j}}{{f}^{-1}}\left( {{x}_{j}} \right)}-{{f}^{-1}}\left( {{x}_{i}} \right) \right|}{2} \right) \right\}} \\ 
 & \le \sum\limits_{i=1}^{n}{{{w}_{i}}{{x}_{i}}}.  
\end{aligned}\]
\end{corollary}
\begin{proof}
Notice that $f(x)=x^p, p\geq 2$ is $2-$radical convex. Since 
\[\prod\limits_{i=1}^{n}{x_{i}^{{{w}_{i}}}}\le \sum\limits_{i=1}^{n}{{{w}_{i}}{{x}_{i}}}\]
and $f$ is increasing, it follows that
\[f\left( \prod\limits_{i=1}^{n}{x_{i}^{{{w}_{i}}}} \right)\le \sum\limits_{i=1}^{n}{{{w}_{i}}\left\{ f\left( \frac{\sum\nolimits_{j=1}^{n}{{{w}_{j}}{{x}_{j}}}+{{x}_{i}}}{2} \right)+f\left( \frac{\left| \sum\nolimits_{j=1}^{n}{{{w}_{j}}{{x}_{j}}}-{{x}_{i}} \right|}{2} \right) \right\}}\le \sum\limits_{i=1}^{n}{{{w}_{i}}f\left( {{x}_{i}} \right)}.\]
Of course $f$ is an invertible function and its inverse is positive. So we may repalce ${{x}_{i}}$ by ${{f}^{-1}}\left( {{x}_{i}} \right)$ in the above inequality. This implies the desired result.
\end{proof}

Next, we show a refinement of the super additivity behavior of convex functions.
\begin{proposition}
Let $f$ be 2-radical convex and let $a,b\geq 0$. Then
\[f\left( a \right)+f\left( b \right)+f\left( \sqrt{2ab} \right)\le f\left( a+b \right).\]
\end{proposition}
\begin{proof}
 Since for any $a,b\in \mathbb{R}$	
	\[{{\left( a+b \right)}^{2}}={{a}^{2}}+{{b}^{2}}+2ab.\]
We have, for $g(x)=f\left(\sqrt{x}\right),$
	\[\begin{aligned}
   g\left( {{\left( a+b \right)}^{2}} \right)&=g\left( {{a}^{2}}+{{b}^{2}}+2ab \right) \\ 
 & \ge g\left( {{a}^{2}}+{{b}^{2}} \right)+g\left( 2ab \right) \\ 
 & \ge g\left( {{a}^{2}} \right)+g\left( {{b}^{2}} \right)+g\left( 2ab \right),  
\end{aligned}\]
where we have used the fact that $g$ is super additive, being a convex function with $g(0)=0$, to obtain the last two inequalities. Noting that $g(x)=f\left(\sqrt{x}\right),$ the proof is complete.
\end{proof}

On the other hand, Hermite-Hadamard inequalities refining \eqref{eq_HH_intro} can be shown as follows.
\begin{theorem}\label{thm_HH}
Let $f$ be 2-radical convex. Then for $b>a>0,$
\begin{align*}
f\left(\frac{a+b}{2}\right)+\frac{2}{b-a}\int_{0}^{\frac{b-a}{2}}f(x)dx\leq \frac{1}{b-a}\int_{a}^{b}f(x)dx,
\end{align*}
and
\begin{align*}
\frac{1}{b-a}\int_{a}^{b}f(x)dx+\frac{1}{b-a}\int_{0}^{\frac{b-a}{2}}\frac{4xf(x)}{\sqrt{(b-a)^2-4x^2}}dx\leq\frac{f(a)+f(b)}{2}.
\end{align*}
\end{theorem}
\begin{proof}
For the first inequality, the Inequality \eqref{8} implies
\begin{align*}
f\left(\frac{a+b}{2}\right)&=f\left(\frac{(1-t)a+tb+ta+(1-t)b)}{2}\right)\\
&\leq \frac{f((1-t)a+tb)+f(ta+(1-t)b)}{2}-f\left(\frac{|(1-t)a+tb-ta-(1-t)b|}{2}\right)\\
&=\frac{f((1-t)a+tb)+f(ta+(1-t)b)}{2}-f\left(\frac{|1-2t|(b-a)}{2}\right).
\end{align*}
Integrating this last inequality over the interval $[0,1]$, noting symmetry of $|1-2t|$ about $t=\frac{1}{2}$ and calculating the integrals
$$\int_{0}^{1}f\left( \left( 1-t \right)a+tb \right)dt=\int_{0}^{1}f(ta+(1-t)b)dt=\frac{1}{b-a}\int_{a}^{b}f(x)dx,$$
and
\begin{align*}
\int_{0}^{1}f\left(\frac{|1-2t|(b-a)}{2}\right)dt&=2\int_{0}^{\frac{1}{2}}f\left(\frac{(1-2t)(b-a)}{2}\right)dt\\
&=\frac{2}{b-a}\int_{0}^{\frac{b-a}{2}}f(x)dx
\end{align*}
imply the first desired inequality.

For the second desired inequality, Corollary \ref{cor_2-conv_ref} implies
$$f\left( \left( 1-t \right)a+tb \right)+f\left( \sqrt {t(1-t)} \left| a-b \right| \right)\le \left( 1-t \right)f\left( a \right)+tf\left( b \right),0\leq t\leq1.$$ Noting that the quantity $ f\left( \sqrt {t(1-t)} \left| a-b \right| \right)$ is symmetric about $\frac{1}{2}$, integrating the above inequality over the interval $[0,1]$ implies
\begin{align}\label{needed_1}
\int_{0}^{1}f\left( \left( 1-t \right)a+tb \right)dt+2\int_{0}^{\frac{1}{2}}f\left( \sqrt {t(1-t)} \left| a-b \right| \right)dt\leq \int_{0}^{1}\left(\left( 1-t \right)f\left( a \right)+tf\left( b \right)\right)dt.
\end{align}
Noting that 
$$\int_{0}^{1}f\left( \left( 1-t \right)a+tb \right)dt=\frac{1}{b-a}\int_{a}^{b}f(x)dx,$$
$$2\int_{0}^{\frac{1}{2}}f\left( \sqrt {t(1-t)} \left| a-b \right| \right)dt=\frac{1}{b-a}\int_{0}^{\frac{b-a}{2}}\frac{4xf(x)}{\sqrt{(b-a)^2-4x^2}}dx; x= \sqrt {t(1-t)} \left| a-b \right|,$$ and
$$ \int_{0}^{1}\left(\left( 1-t \right)f\left( a \right)+tf\left( b \right)\right)dt=\frac{f(a)+f(b)}{2},$$
implies the desired inequality.
\end{proof}

Further, we have the following integral inequality, as a special case of \eqref{eq_HH_intro}. The general case is stated in Theorem \ref{thm_HH_p} below.
\begin{theorem}\label{thm_int}
Let $f$ be 2-radical convex. Then 
$$f\left(\frac{1}{2}\right)\leq \int_{0}^{1}\left\{f\left(\frac{x+\frac{1}{2}}{2}\right)+f\left(\frac{|x-\frac{1}{2}|}{2}\right)\right\}dx\leq \int_{0}^{1}f(x)dx.$$
\end{theorem}
\begin{proof}
  From Theorem \ref{10}, \[f\left( \sum\limits_{i=1}^{n}{{{w}_{i}}{{x}_{i}}} \right)\le \sum\limits_{i=1}^{n}{{{w}_{i}}f\left( \frac{\sum\nolimits_{j=1}^{n}{{{w}_{j}}{{x}_{j}}}+{{x}_{i}}}{2} \right)}+\sum\limits_{i=1}^{n}{{{w}_{i}}f\left( \frac{\left| \sum\nolimits_{j=1}^{n}{{{w}_{j}}{{x}_{j}}}-{{x}_{i}} \right|}{2} \right)}\le \sum\limits_{i=1}^{n}{{{w}_{i}}f\left( {{x}_{i}} \right)},\]
for any positive $w_i$'s with $\sum_{i=1}^{n}w_i=1$ and any $x_i\geq 0.$ In particular, for $n\in\mathbb{N}$, let $w_i=\frac{1}{n}$ and $w_i=\frac{i}{n}.$
Since $f$ is convex non-negative and $f(0)=0,$ it follows that $f$ is increasing. Therefore,  Riemann sums entail the following two integrals
$$\lim_{n\to\infty}\sum_{i=1}^{n}w_ix_i= \int_{0}^{1}xdx=\frac{1}{2}$$ and
$$\lim_{n\to\infty}\sum_{i=1}^{n}w_if(x_i)=\int_{0}^{1}f(x)dx.$$ To complete the proof of the theorem, it remains to show that
\begin{align}\label{needed_2}
&\lim_{n\to\infty}\left\{\sum\limits_{i=1}^{n}{{{w}_{i}}f\left( \frac{\sum\nolimits_{j=1}^{n}{{{w}_{j}}{{x}_{j}}}+{{x}_{i}}}{2} \right)}+\sum\limits_{i=1}^{n}{{{w}_{i}}f\left( \frac{\left| \sum\nolimits_{j=1}^{n}{{{w}_{j}}{{x}_{j}}}-{{x}_{i}} \right|}{2} \right)}\right\}\notag\\
&=\int_{0}^{1}\left\{f\left(\frac{x+\frac{1}{2}}{2}\right)+f\left(\frac{|x-\frac{1}{2}|}{2}\right)\right\}dx.
\end{align}
Since $\lim_{n\to\infty}\sum_{i=1}^{n}w_ix_i=\frac{1}{2},$ and $f$ is continuous, then given a positive number $\epsilon$, there exists $n_{\epsilon}\in\mathbb{N}$ such that 

$$f\left(\frac{\frac{1}{2}+x_i}{2}\right)-\epsilon<f\left( \frac{\sum\nolimits_{j=1}^{n}{{{w}_{j}}{{x}_{j}}}+{{x}_{i}}}{2} \right)<f\left(\frac{\frac{1}{2}+x_i}{2}\right)+\epsilon, \forall n\geq n_{\epsilon}.$$

Therefore, for $ n\geq n_{\epsilon}$,
\begin{align}
\sum_{i=1}^{n}w_i f\left(\frac{\frac{1}{2}+x_i}{2}\right)-\epsilon<\sum_{i=1}^{n}w_i f\left( \frac{\sum\nolimits_{j=1}^{n}{{{w}_{j}}{{x}_{j}}}+{{x}_{i}}}{2} \right)<\sum_{i=1}^{n}w_i f\left(\frac{\frac{1}{2}+x_i}{2}\right)+\epsilon.\label{needed_3}
\end{align}
Clearly,
$$\lim_{n\to\infty}\sum_{i=1}^{n}w_if\left(\frac{\frac{1}{2}+x_i}{2}\right)=\int_{0}^{1}f\left(\frac{\frac{1}{2}+x}{2}\right) dx.$$
Consequently, \eqref{needed_3} implies, for arbitrarily small $\epsilon>0$,
\begin{align}
\int_{0}^{1}f\left(\frac{\frac{1}{2}+x}{2}\right) dx-\epsilon\leq \limsup_n 
\sum\limits_{i=1}^{n}{{{w}_{i}}f\left( \frac{\sum\nolimits_{j=1}^{n}{{{w}_{j}}{{x}_{j}}}+{{x}_{i}}}{2} \right)}\leq \int_{0}^{1}f\left(\frac{\frac{1}{2}+x}{2}\right) dx+\epsilon,\label{needed_4}
\end{align}
Letting $\epsilon\to 0^+$ implies that
$$\lim_{n\to\infty}
\sum\limits_{i=1}^{n}{{{w}_{i}}f\left( \frac{\sum\nolimits_{j=1}^{n}{{{w}_{j}}{{x}_{j}}}+{{x}_{i}}}{2} \right)}=\int_{0}^{1}f\left(\frac{\frac{1}{2}+x}{2}\right) dx.$$

A similar argument implies that
$$\lim_{n\to\infty}\sum\limits_{i=1}^{n}{{{w}_{i}}f\left( \frac{\left| \sum\nolimits_{j=1}^{n}{{{w}_{j}}{{x}_{j}}}-{{x}_{i}} \right|}{2} \right)}=\int_{0}^{1} f\left(\frac{|x-\frac{1}{2}|}{2}\right)dx.$$
The last two identities together with \eqref{needed_2} imply the desired result.
\end{proof}
An unexpected property of 2-radical convex functions that follows from Theorem \ref{thm_int} is the following integrals bounds .
\begin{corollary}\label{cor_int}
Let $f$ be 2-radical convex. Then
$$ 3\int_{0}^{\frac{1}{4}}f(x)dx+\int_{\frac{1}{4}}^{\frac{3}{4}}f(x)dx\leq \int_{\frac{3}{4}}^{1}f(x)dx.$$
\end{corollary}
\begin{proof}
From Theorem \ref{thm_int}, we have
$$ \int_{0}^{1}\left\{f\left(\frac{x+\frac{1}{2}}{2}\right)+f\left(\frac{|x-\frac{1}{2}|}{2}\right)\right\}dx\leq \int_{0}^{1}f(x)dx.$$
Noting that 
$$\int_{0}^{1}f\left(\frac{|x-\frac{1}{2}|}{2}\right)dx=\int_{0}^{\frac{1}{2}}f\left(\frac{\frac{1}{2}-x}{2}\right)dx+\int_{\frac{1}{2}}^{1}f\left(\frac{x-\frac{1}{2}}{2}\right)dx,$$
then substituting $\frac{x+\frac{1}{2}}{2}=y, \frac{\frac{1}{2}-x}{2}=z, \frac{x-\frac{1}{2}}{2}=w$ imply the desired inequality.
\end{proof}
What Corollary \ref{cor_int} says is that, on average, the values of a 2-radical convex function on the interval $\left[\frac{3}{4},1\right]$ are much bigger than its values on the interval $\left[0,\frac{3}{4}\right].$\\
Numerical examples show these differences!

In fact, the proof of Theorem \ref{thm_int} can be carried out over any interval $[a,b]$ to obtain the following natural generalization of the Hermite-Hadamard inequality \eqref{eq_HH_intro}.

\begin{theorem}\label{thm_HH_p}
Let $f$ be $2-$radical convex and let $b,a> 0.$ Then 
\begin{align}\label{eq_HH_p-conv}
f\left(\frac{a+b}{2}\right)\leq \frac{1}{b-a}\int_{a}^{b}\left\{f\left(\frac{x+\frac{a+b}{2}}{2}\right)+f\left(\frac{|x-\frac{a+b}{2}|}{2}\right)\right\}dx\leq \frac{1}{b-a} \int_{a}^{b}f(x)dx.
\end{align}
\end{theorem}
\begin{proof}
Let $f$ be $p-$radical convex and assume $b>a.$ If $w_i=\frac{1}{n}$ and $x_i=a+\frac{b-a}{n}i,$ then
$$\sum_{i=1}^{n}w_ix_i=\frac{1}{b-a}\sum_{i=1}^{n}\frac{b-a}{n}x_i\to \frac{1}{b-a}\int_{a}^{b}xdx=\frac{b+a}{2}.$$ Since $f$ is continuous, it follows that $$f\left(\sum_{i=1}^{n}w_ix_i\right)\to f\left(\frac{a+b}{2}\right).$$ This implies the left hand side of \eqref{eq_HH_p-conv}. The middle and right sides of \eqref{eq_HH_p-conv} follow similarly, adopting similar proof to Theorem \ref{thm_int}.
\end{proof}

Another application of $2$-radical convex functions is the following new form of the continuous Jensen inequality \eqref{eq_jensen_cont_intro}.
\begin{corollary}
Let $f$ be 2-radical convex and let $g:[a,b]\to [0,\infty)$ be  continuous. Then
\begin{align*}
f\left(\frac{1}{b-a}\int_{a}^{b}g(x)dx\right)&\leq \frac{1}{b-a}\int_{a}^{b}\left\{f\left(\frac{x+\frac{1}{b-a}\int_{a}^{b}g(x)dx}{2}\right)+f\left(\frac{|x-\frac{1}{b-a}\int_{a}^{b}g(x)dx|}{2}\right)\right\}dx\\
&\leq \frac{1}{b-a}\int_{0}^{1}f(g(x))dx.
\end{align*} 
\end{corollary}
\begin{proof}
Replacing $x_i$ in Theorem \ref{10} by $g(x_i)$, then proceeding like Theorems \ref{thm_int} and \ref{thm_HH_p} imply the desired inequalities. 
\end{proof}
\section{More general discussion}
As mentioned in the introduction, one of our goals is to show how convex functions behave differently, due to their convexity behavior.

Our first result in this section is the following necessary condition following from $p-$radical convexity.
\begin{proposition}\label{prop_p-conv_avg}
Let $f$ be $p-$radical convex, for some $p\geq 1$. Then, for every $x\geq 0,$ 
\begin{equation}\label{needed_0}
\int_{0}^{x}f(t)dt\leq \frac{x}{p+1}f(x).
\end{equation}
Equality in \eqref{needed_0} holds, for all $x>0$, if and only if $f(x)=cx^p,$ for some constant $c.$
\end{proposition}
\begin{proof}
Assume first that $f$ is twice differentiable and let $g(x)=f\left(x^{\frac{1}{p}}\right).$ Since $g$ is convex, $f$ being $p-$radical convex, it follows that $g''\geq 0.$ This implies that $(1-p)f'\left(x^{\frac{1}{p}}\right)+x^{\frac{1}{p}}f''\left(x^{\frac{1}{p}}\right)\geq 0$ for all $x> 0.$ This implies that
$xf''(x)\geq (p-1)f'(x), $ for all $x> 0.$ Integrating this inequality on $[0,x]$ twice by parts implies the desired inequality, when $f$ is twice differentiable.\\
When $f$ is not twice differentiable, let $g(x)=f\left(x^{\frac{1}{p}}\right)$ and let $g_n$ be a sequence of twice differentiable convex functions such that $g_n\to g$ uniformly. Such a sequence can be found using \cite[Theorem 1]{azagra}. Let $h_n(x)=g_n(x^p).$ Since $g_n$ is convex, it follows that $h_n$ is convex, because $p\geq 1.$ Further, $h_n$ is twice differentiable and $h_n\left(x^{\frac{1}{p}}\right)=g_n(x),$ which is convex. That is, $h_n$ is a sequence of $p-$radical convex twice differentiable functions  such that $h_n\to f$ uniformly. Since \eqref{needed_0} is valid for $p-$radical convex twice differentiable functions, it is valid for $h_n$, and hence
$$\int_{0}^{x}h_n(t)dt\leq \frac{x}{p+1}h_n(x),\;\forall x>0.$$ Letting $n\to\infty$ and noting that $h_n\to f$ uniformly on the compact interval $[0,x]$ imply $$\int_{0}^{x}f(t)dt\leq \frac{x}{p+1}f(x).$$ This completes the proof of \eqref{needed_0}. For the equality condition, direct computations show that $f(x)=cx^p$ turns \eqref{needed_0} into an equality. Also, assuming equality in \eqref{needed_0} and differentiating, we obtain $f(x)=\frac{1}{p+1}(f(x)+xf'(x)).$ Solving this differential equation implies that $f(x)=cx^p$, for some constant $c$. This completes the proof.

\end{proof}

Our first observation about Proposition \ref{prop_p-conv_avg} is that a function cannot be $p-$radical convex, for all $p\geq 1,$ as we show next.
\begin{corollary}
A function $f:[0,\infty)\to [0,\infty)$ is $p-$radical convex for all $p\geq 1$ if and only if $f=0$.
\end{corollary}
\begin{proof}
 Assume that $f$ is $p-$radical convex for all $p\geq 1.$ Then Proposition \ref{prop_p-conv_avg} implies
$$\int_{0}^{x}f(t)dt\leq \frac{x}{p+1}f(x), \forall x>0, \forall p\geq 1.$$ Letting $p\to\infty$ and noting that $f\geq 0$ implies that
$$\int_{0}^{x}f(t)dt=0,\;\forall x>0,$$ which gives $f(x)=0,\;\forall x,$ upon differentiation. This completes the proof. 
\end{proof}

The inequality \eqref{needed_0} can be used sometimes to decide if a function is not $p-$radical convex. For example, consider the function $f(x)=e^x-1.$ This function is convex, increasing and $f(0)=0.$ Calculating
$$\int_{0}^{x}(e^t-1)dt=e^x-x-1\;{\text{and}}\;\frac{x}{3}f(x)=\frac{x(e^x-1)}{3}.$$ It is then clear that when $x=1$, we have 
$$e^x-x-1< \frac{x(e^x-1)}{3},$$ while the inequality is reversed when $x=3.$ This shows that the function $f(x)=e^x-1$ does not satisfy \eqref{needed_0} for all $x$, and hence it is not $2-$radical convex. However, it should be noted that $f(x)=e^x-1-x$ is $2-$radical convex.

Also, we notice that \eqref{needed_0} can be written as
\begin{equation}\label{eq_p_conv_avg}
\frac{1}{x}\int_{0}^{x}f(t)dt\leq \frac{1}{p+1}f(x).
\end{equation}
The left side of this inequality is the average value of $f$ over the interval $[0,x].$ The Hermite-Hadamard inequality assures that (when $f(0)=0$)
$$\frac{1}{x}\int_{0}^{x}f(t)dt\leq \frac{1}{2}f(x).$$ This shows that $p-$radical convex functions have tighter bounds than convex functions.

One more observation about \eqref{eq_p_conv_avg} is its similarity to the well known Hardy inequality \eqref{eq_hardy_intro}. We notice first that a convex function on $[0,\infty)$ with $f(0)=0$ is not an $L^p$ function, for any $p\geq 1.$ However, we have the following new Hardy-type inequality.
\begin{theorem}\label{thm_hardy}
Let $f$ be $p-$radical convex, and let $\alpha,\beta>0.$ Then, for $p\geq 1,$
\begin{equation}\label{eq_hardy}
\int_{\alpha}^{\beta}\left(\frac{1}{x}\int_{0}^{x}f(t)dt\right)^pdx\leq \left(\frac{1}{p+1}\right)^{p}\int_{\alpha}^{\beta}f(x)^pdx.
\end{equation}
The inequality is sharp, and the function $f(x)=x^p$ turns this inequality to an identity.
\end{theorem}
\begin{proof}
The proof follows immediately from Proposition \ref{prop_p-conv_avg}.
\end{proof}

Further, we have the following better bound for $m$-radical convex functions, when $m\geq 2$ is an even integer.

\begin{theorem}\label{thm_m-convex}
Let $f$ be $m-$radical convex for some even integer $m\geq 2.$ Then
\begin{align*}
(1-t)f(a)+tf(b)&\geq\sum_{k=0}^{\frac{m}{2}}f\left({{m/2}\choose k}^{\frac{1}{m}}(t(1-t))^{\frac{k}{m}}(a-b)^{\frac{2k}{m}}((1-t)a+tb)^{1-\frac{2k}{m}}\right)\\
&= f((1-t)a+tb)+\sum_{k=1}^{\frac{m}{2}}f\left({{m/2}\choose k}^{\frac{1}{m}}(t(1-t))^{\frac{k}{m}}(a-b)^{\frac{2k}{m}}((1-t)a+tb)^{1-\frac{2k}{m}}\right)
\end{align*}
where $a,b>0$ and $0\leq t\leq 1.$ In particular,
$$\frac{f(a)+f(b)}{2}\geq \sum_{k=0}^{\frac{m}{2}}f\left({{m/2} \choose k}^{\frac{1}{m}} \left(\frac{a+b}{2}\right)^{\frac{2k}{m}}\left(\frac{a-b}{2}\right)^{1-\frac{2k}{m}}    \right).$$

\end{theorem}
\begin{proof}
Let $f$ be $m-$radical convex, and define $g(x)=f(\sqrt[m]{x}).$ Then, by definition, $g$ is  convex and $g(0)=0.$ Consequently,
\begin{align*}
(1-t)f(a)+tf(b)&=(1-t)g(a^m)+tg(b^m)\\
&\geq g((1-t)a^m+tb^m)\\
&\geq g\left([(1-t)a^2+tb^2]^{\frac{m}{2}}\right)\\
&=g\left([t(1-t)(a-b)^2+((1-t)a+tb)^2]^{\frac{m}{2}}\right)\\
&=g\left( \sum_{k=0}^{\frac{m}{2}}{ {m/2}\choose k}(t(1-t)(a-b)^2)^k((1-t)a+tb)^2)^{\frac{m}{2}-k}                       \right)\\
&\geq  \sum_{k=0}^{\frac{m}{2}}g\left({ {m/2}\choose k}(t(1-t)(a-b)^2)^k((1-t)a+tb)^2)^{\frac{m}{2}-k}                       \right)\\
&=\sum_{k=0}^{\frac{m}{2}}f\left({ {m/2}\choose k}^{\frac{1}{m}}(t(1-t))^{\frac{k}{m}}(a-b)^{\frac{2k}{m}}((1-t)a+tb)^{1-\frac{2k}{m}}                       \right),
\end{align*}
where we have used the facts that $g$ is convex and super additive to obtain the first and third inequalities, respectively, while the facts that the function $t\mapsto t^{\frac{m}{2}}$ is convex and $g$ is increasing were used to obtain the second inequality in the above computations. This completes the proof.
\end{proof}

For example, when $f$ is $4-$radical convex, applying Theorem \ref{thm_m-convex} implies the following.

\begin{corollary}\label{cor_4-conv}
Let $f$ be 4-radical convex, and let $a,b\geq 0$. Then,
\[\begin{aligned}
  & f\left( \left( 1-t \right)a+tb \right)+f\left( \sqrt{t\left( 1-t \right)}\left| a-b \right| \right)+f\left( \sqrt[4]{2t\left( 1-t \right)}\sqrt{\left| a-b \right|\left( \left( 1-t \right)a+tb \right)} \right) \\ 
 & \le \left( 1-t \right)f\left( a \right)+tf\left( b \right), 0\leq t\leq 1.  
\end{aligned}\]
In particular,
\[f\left( \frac{a+b}{2} \right)+f\left( \frac{\left| a-b \right|}{2} \right)+f\left( \frac{1}{{{2}^{\frac{3}{4}}}}\sqrt{\left| a-b \right|\left( a+b \right)} \right)\le \frac{f\left( a \right)+f\left( b \right)}{2}.\]
\end{corollary}
We conclude this article by emphasizing that although $p-$radical convex functions are convex, treating them as $p-$radical convex implies better convex inequalities. Further, the largest $p$ such that $f$ is $p-$radical convex implies the best bound in our inequalities. For example, when $f$ is $4-$radical convex, it is $2-$radical convex. However, applying Corollary \ref{cor_4-conv} implies better bounds than Corollary \ref{cor_2-conv_ref}.

{\tiny \vskip 0.3 true cm }
{\tiny (M. Sababheh) Department of Basic Sciences, Princess Sumaya University For Technology, Al Jubaiha, Amman 11941, Jordan.}

{\tiny \textit{E-mail address:} sababheh@psut.edu.jo}
{\tiny \vskip 0.3 true cm }

{\tiny (H. R. Moradi) Department of Mathematics, Payame Noor University (PNU), P.O. Box 19395-4697, Tehran, Iran.}

{\tiny \textit{E-mail address:} hrmoradi@mshdiau.ac.ir }
%-----------------------------------------------------------------------------
%-----------------------------------------------------------------------------
\end{document}